\documentclass[11pt]{article}
\usepackage{amsmath,amssymb,amsthm,amscd,amsfonts,enumerate,mathrsfs}

\newtheorem{theorem}{Theorem}[subsection]
\newtheorem{proposition}[theorem]{Proposition}
\newtheorem{lemma}[theorem]{Lemma}
\newtheorem{corollary}[theorem]{Corollary}

\theoremstyle{definition}
\newtheorem{definition}[theorem]{Definition}

\newtheorem{problem}[theorem]{Problem}
\newtheorem{example}[theorem]{Example}

\newtheorem{conjecture}[theorem]{Conjecture}

\newcommand{\N}{\mathbb{N}}
\newcommand{\Z}{\mathbb{Z}}

\newcommand{\R}{\mathbb{R}}
\newcommand{\C}{\mathbb{C}}
\newcommand{\K}{\mathbb{K}}
\renewcommand{\H}{\mathcal{H}}
\newcommand{\M}{\mathcal{M}}
\newcommand{\B}{\mathcal{B}}
\newcommand{\A}{\mathcal{A}}
\renewcommand{\O}{\mathcal{O}}

\newcommand{\abs}[1]{\left\lvert #1 \right\rvert}
\newcommand{\norm}[1]{\left\lVert #1 \right\rVert}
\newcommand{\innpr}[2]{\left\langle #1, #2 \right\rangle}

\DeclareMathOperator{\id}{id}
\DeclareMathOperator{\orb}{orb}
\DeclareMathOperator{\dist}{d}

\DeclareMathOperator{\Span}{span}

\title{Application of Operator Theory for the Collatz Conjecture}
\author{Takehiko Mori\footnote{ORCID:0000-0001-6448-8293}\\
Graduate School of Science and Engineering\\
Chiba University \\
Inage-ku, Chiba 263-8522, Japan
}

\date{}

\begin{document}

\maketitle

\begin{abstract}


The Collatz map (or the $3n{+}1$-map) $f$ is defined on positive integers by setting $f(n)$ equal to $3n+1$ when $n$ is odd and $n/2$ when $n$ is even. The Collatz conjecture states that starting from any positive integer $n$, some iterate of $f$ takes value $1$. In this study, we discuss formulations of the Collatz conjecture by $C^{*}$-algebras in the following three ways: (1) single operator, (2) two operators, and  (3) Cuntz algebra. For the $C^{*}$-algebra generated by each of these, we consider the condition that it has no non-trivial reducing subspaces. For (1), we prove that the condition implies the Collatz conjecture. In the cases (2) and (3), we prove that the condition is equivalent to the Collatz conjecture. For similar maps, we introduce equivalence relations by them and generalize connections between the Collatz conjecture and irreducibility of associated $C^{*}$-algebras.

\end{abstract}

\section{Introduction}

The Collatz conjecture (or the $3n{+}1$-problem) is a longstanding open problem for positive integers which is named after Lother Collatz who described the problem in an informal lecture in 1950 at the International Math.\ Congress in Cambridge, Massachusetts \cite{MR2560718}.
It is also known as Syracuse problem, Hasse's algorithm, Kakutani's problem and Ulam's problem.
The Collatz map $f$ is defined on positive integers by setting $f(n)$ equal to $3n+1$ when $n$ is odd and $n/2$ when $n$ is even. The problem is to prove that starting from any positive integer $n$, some iterate of $f$ takes value $1$.
According to a recent study, the conjecture has been shown to hold for all starting values up to $10^{20}$ by computer \cite{Ba21JS}.
At present, the Collatz conjecture remains unsolved.

The interaction between dynamical systems and operator theory has long been an active area of study. Especially there have been many interesting works which link dynamical systems to the theory of $C^*$-algebras. A $C^*$-algebra is a $*$-algebra consisting of operators on Hilbert spaces equipped with the operator norm topology.
In their pioneering work \cite{GPS95JRAM}, T. Giordano, I. F. Putnam, and C. F. Skau showed that two minimal homeomorphisms on the Cantor set are strongly orbit equivalent if and only if the associated crossed product $C^*$-algebras are isomorphic.
Later, this was generalized to minimal $\Z^d$-actions \cite{GMPS10IM}.
M. Boyle and J. Tomiyama studied relationships between orbit equivalence for topologically free homeomorphisms on compact Hausdorff spaces and the associated crossed product $C^*$-algebras \cite{BT98JMSJ}.
Another important example is the Cuntz-Krieger algebras \cite{CK80IM} arising from topological Markov shifts, and they have been extensively studied by many hands.
In \cite{Ma97IJM} and \cite{Ma98MS}, K. Matsumoto introduced the $C^*$-algebras associated with general subshifts and computed their $K$-groups. Matsumoto and Matui \cite{MM14KJM} classified topological Markov shifts up to continuous orbit equivalence based on the idea of Cuntz-Krieger algebras.

In this paper, we present new relationship between dynamical systems and $C^{*}$-algebras. Namely, we will formulate the Collatz conjecture in terms of $C^*$-algebras.
There has never been a study which provides connection between the Collatz conjecture and operator algebras on Hilbert spaces, and the present paper is a first attempt for such a direction.
We discuss formulations of the Collatz conjecture by $C^{*}$-algebras in the following three ways: (1) single operator, (2) two operators, and  (3) Cuntz algebra. For the $C^{*}$-algebra generated by each of these, we consider the condition that it has no non-trivial reducing subspaces. For (1), we prove that the condition implies the Collatz conjecture (Corollary \ref{cor:SingleS}). In the cases (2) and (3), we prove that the condition is equivalent to the Collatz conjecture (Theorem \ref{thm:TwoNS} and Theorem \ref{thm:CuntzNS}). 
We remark that J. Leventides and C. Poulios \cite{LP21IFACPO} studied the Collatz conjecture by functional analytic approach. They associated with the Collatz map a bounded operator $K$ on the space $l_1(\N)$ of all absolutely summable sequences of real numbers on $\N$, and rephrased the Collatz conjecture in terms of the operator $K$. S. Letherman, D. Schleicher, and R. Wood \cite{LSW99EM} studied the Collatz conjecture by using holomorphic functions.  They constructed entire holomorphic functions which are extensions of the Collatz map, and generalized the Collatz conjecture for the holomorphic functions. The present paper can be said to be in the trend of these results. 

The paper is organized as follows.
Section \ref{sec:Collatz conjecture} recalls the concept of orbits in dynamical systems and state the Collatz conjecture in that context.
The section closes with definition of orbit equivalence relations by maps.
This is used in Section \ref{sec:Formulation of the Collatz conjecture} and \ref{sec:Generalization}.
Section \ref{sec:Hilbert spaces} reviews the theory of Hilbert spaces and $C^{*}$-algebras for readers who are not familiar with functional analysis and operator theory.
Section \ref{sec:Formulation of the Collatz conjecture} discusses formulations of the Collatz conjecture as operator theoretic problems. We define some operators on a Hilbert space and $C^{*}$-algebras generated by these operators. The properties of these $C^{*}$-algebras are examined and proved in this section.
Section \ref{sec:Generalization} generalizes study in Section \ref{sec:Formulation of the Collatz conjecture}. Bounded and separating condition of maps are defined in this section which characterize generalization of connections between the Collatz conjecture and irreducibility of associated $C^{*}$-algebras.

Throughout this paper, $\N$ denotes the set of all positive integers and $\abs{x}$ denotes the absolute value of a number $x$. For a topological space $X$ and its subspace $Y$, $\overline{Y}$ denotes the closure of $Y$ in $X$. For $k \in \N_{>1}$, $\{1,2,\cdots,k\}^*$ denotes the disjoint union over $n \in \N$ of $\{1,2,\cdots,k\}^n$.

\section{The Collatz conjecture and dynamical systems} \label{sec:Collatz conjecture}

In this section, we consider the iteration of the Collatz map and describe the Collatz conjecture as a problem on discrete dynamical systems. For more developed studies on this point of view, see e.g., \cite{LSW99EM} and \cite{LP21IFACPO}.

\subsection{Dynamical systems}

Let $X$ be a set and $f:X \to X$ be a map on $X$. Let us denote $f^n$ as the $n$-th iterate of $f$, where $n \in \{0\} \cup \N$, that is,
\begin{equation*}
f^0=\id_X
\end{equation*}
and
\begin{equation*}
f^{n+1}=f \circ f^{n}.
\end{equation*}

Let us recall the notion of orbits and first-return maps.
\begin{definition}[Orbits]
For $x \in X$, the (forward) orbit of $x$, denoted by $\orb(x;f)$, is defined as
\begin{equation*}
\orb(x;f)=\{f^n(x) \mid n \in \{0\} \cup \N\}.
\end{equation*}
\end{definition}

\begin{definition}[First-return maps] \label{def:First-return map}
Let $\Sigma$ $(\neq \emptyset) \subseteq X$ and $x \in \Sigma$. If there exists $n \in \N$ such that $f^n(x) \in \Sigma$, then we will let $\tau(x)$ denote the minimum positive integer as such. We will let $\Sigma'$ denote the set of all $x \in X$ such that there exists the $\tau(x)$. Then, the first return map (or Poincar\'{e} map) for $f$ on $\Sigma$ is the map $P:\Sigma' \to \Sigma$ defined by:
\begin{equation*}
P(x)=f^{\tau(x)}(x), \text{ where } x \in \Sigma'.
\end{equation*}
\end{definition}

\subsection{The Collatz conjecture}

\begin{definition}[The Collatz map]
The Collatz map $f:\N \to \N$ is defined by:
\begin{equation*}
f(n)=
\begin{cases}
3n+1, & n:\text{odd},\\
n/2, & n:\text{even}.
\end{cases}
\end{equation*}
\end{definition}
For the Collatz map $f$, the following statement is well-known as the Collatz conjecture or the $3n{+}1$-problem.
\begin{conjecture}[The Collatz conjecture]
For every $n \in \N$, there is $m \in \{0\} \cup \N$ such that $f^m(n)=1$, that is, $1 \in \orb(n;f)$ for every $n \in \N$.
\end{conjecture}

\subsection{Equivalence relations}

Let $X$ be a set and $f:X \to X$ be a map. We define an equivalence relation on $X$ by $f$.

\begin{definition}

We define a binary relation on $X$ which is called the orbit equivalence relation by $f$ as follows:
For $x,y \in X$,

\begin{equation*}
x \sim y \iff \orb(x;f) \cap \orb(y;f) \neq \emptyset,
\end{equation*}
in other words, there exist $n,m \in \{0\} \cup \N$ such that $f^{n+k}(x)=f^{m+k}(y)$ for all $k \in \{0\} \cup \N$.

\end{definition}
It is easy to check that the orbit equivalence relation by $f$ is an equivalence relation.
\begin{lemma}[Equivalence relations]
If $\sim$ is defined as above, then $\sim$ is an equivalence relation.
\end{lemma}

\section{Hilbert spaces and operator theory} \label{sec:Hilbert spaces}

In this section, we will see some fundamental definitions and theorems in operator theory. For more details, see, e.g., \cite{Co00GSM} and \cite{Co96GTM}.

$\K$ will denote either the real field, $\R$, or the complex field, $\C$. For any set $X$, the Kronecker delta on $X$, denoted by $\delta$, is defined as follows:
For $x,y \in X$,
\begin{equation*}
\delta_{x,y}=
\begin{cases}
1, & x=y,\\
0, & x \neq y.
\end{cases}
\end{equation*}

\subsection{Hilbert spaces}

Let $\H$ be an inner product space on $\K$, i.e., $\H$ is a vector space over $\K$ and there is an inner product $\innpr{\cdot}{\cdot}:\H \times \H \to \K$, and distance function $\dist$ on $\H$ be defined by:

\begin{equation*}
\dist(x,y)=\abs{\innpr{x-y}{x-y}}^{1/2}, \text{ where } x,y \in \H.
\end{equation*}

Let us define Hilbert spaces on $\K$ as follows.

\begin{definition}[Hilbert spaces]
An inner product space $\H$ is a Hilbet space on $\K$ when $\H$ is complete for the topology induced by $\dist$ as defined above.
\end{definition}

From now on, we will assume that $\H$ is a Hilbert space on $\K$ and $\innpr{\cdot}{\cdot}$ is the inner product on $\H$.
Let us recall the notion of orthogonal complements and completely orthonormal systems.

\begin{definition}[Orthogonal complements]
Let $\M(\neq \emptyset) \subseteq \H$. The orthogonal complement for $\M$, denoted by $\M^\bot$, is defined by:
\begin{equation*}
\M^\bot=\{y \in \H \mid \innpr{x}{y}=0, \forall x \in \M\}.
\end{equation*}
\end{definition}

\begin{theorem}[{\cite[Chapter I, Theorem 2.5, Theorem 2.6]{Co96GTM}}, Orthogonal decomposition] \label{thm:orthogonal decomposition}
Let $\M \subseteq \H$ be a closed subspace. For every $x \in \H$, there are unique $y \in \M$ and $z \in \M^\bot$ such that $x=y+z$.
\end{theorem}

\begin{definition}[Completely orthonormal systems]
Let $\{e_\lambda\}_{\lambda \in \Lambda} \subseteq \H$ be a set, where $\Lambda$ is an index set. The set $\{e_\lambda\}_{\lambda \in \Lambda}$ is a completely orthonormal system (abbreviated as C.O.N.S.) for $\H$, when the following conditions hold:
\begin{enumerate}[(i)]
\item For every $\lambda$ and $\mu \in \Lambda$, $\innpr{e_\lambda}{e_\mu}=\delta_{\lambda , \mu}$;
\item For $x \in \H$, $x=0 \iff \innpr{x}{e_\lambda}=0$ for every $\lambda \in \Lambda$.
\end{enumerate}
\end{definition}
Let us state the following theorem without the proof. For the proof of the theorem, see \cite{Co96GTM}.
\begin{theorem}[{\cite[Chapter I, Theorem 4.13, Proposition 4.14]{Co96GTM}}, Fourier series] \label{thm:Fourier series}
For every Hilbert space $\H$ on $\K$, the following hold:
\begin{enumerate}[(i)]
\item There is a C.O.N.S. for $\H$ and its cardinality depends on $\H$ only; 
\item If $\{e_\lambda\}_{\lambda \in \Lambda}$ is a C.O.N.S. for $\H$, then every $x \in \H$ can be expanded \\by $\{ \langle x,e_\lambda \rangle e_\lambda \}_{\lambda \in \Lambda}$, i.e., $x$ is a unique accumulation point of the set \\$\{\sum_{\lambda \in F}\innpr{x}{e_\lambda}e_\lambda \mid F\text{ is a finite subset of }\Lambda \}$.
\end{enumerate}
\end{theorem}

From the first statement above, we will define the dimension of $\H$ on $\K$, denoted by $\dim{\H}$, as the cardinality of a C.O.N.S. for $\H$. For the second statement, we will write $x=\sum_{\lambda \in \Lambda}\innpr{x}{e_\lambda}e_\lambda$ for every $x \in \H$.

\subsection{Bounded linear operators}

When we define $\norm{\cdot}:\H \to \R$ by:
\begin{equation*}
\norm{x}=\abs{\innpr{x}{x}}^{1/2}, \text{ where }x \in \H,
\end{equation*}
$\norm{\cdot}$ is a norm on $\H$. Thus we can assume that $\H$ is a normed space and discuss bounded linear operators on $\H$.

\begin{definition}[Bounded linear operators]
Let $\norm{\cdot}$ be defined as above and $T:\H \to \H$ be a linear map. We say that $T$ is bounded if
\begin{equation*}
\sup\{ \norm{Tx} \mid x \in \H, \norm{x} \le 1 \}<+\infty.
\end{equation*}
The set of all bounded linear operators on $\H$ is denoted by $\B (\H)$.
\end{definition}

For example, the identity map on $\H$ is a bounded linear operator. We will denote this map by $I$.

Next, we define the adjoint of a bounded linear operator which is also bounded.
\begin{definition}[Adjoint of an operator]
Let $T \in \B(\H)$. The adjoint of $T$, denoted by $T^*$, is the bounded linear operator on $\H$ such that
\begin{equation*}
\innpr{Tx}{y}=\innpr{x}{T^*y}
\end{equation*}
for every $x, y \in \H$.
\end{definition}

For every $T \in \B(\H)$, the adjoint of $T$ exists uniquely (\cite[Chapter I, Theorem 2.2]{Co96GTM}).
It is easy to see that $I^{*}=I$. More generally, there are many operators $T \in \B(\H)$ satisfying that $T^{*}=T$.

By using the definition of the adjoint of an operator, we introduce the notion of orthogonal projections.
\begin{definition}[Orthogonal projections]
We say that $Q \in \B(\H)$ is an orthogonal projection (or projection) when $Q=Q^2$ and $Q=Q^*$.
\end{definition}

Finally, we define isometries and partial isometries on $\H$.
\begin{definition}[Isometries and partial isometries]
We say that $T \in \B(\H)$ is an isometry when
\begin{equation*}
\norm{Tx}=\norm{x}
\end{equation*}
for every $x \in \H$. And $T$ is a partial isometry when there is a closed subspace $\M \subseteq \H$ such that
\begin{equation*}
\norm{Tx}=\norm{x}
\end{equation*}
and
\begin{equation*}
Ty=0
\end{equation*}
for every $x \in \M$ and $y \in \M^\bot$.
\end{definition}

\subsection{Reducing subspaces}

In this subsection, we recall the notion of invariant subspaces and reducing subspaces.
Let $T \in \B(\H)$ and $\M$ is a subspace of $\H$.
\begin{definition}[Invariant subspaces]
We say that $\M$ is an invariant subspace for $T$ when $T\M=\{Tx \mid x \in \M\} \subseteq \M$.
\end{definition}

\begin{definition}[Reducing subspaces]
We say that $\M$ is a reducing subspace for $T$ when $T\M \subseteq \M$ and $T^*\M \subseteq \M$.
\end{definition}

If $\M=\{0\}$ or $\H$, then $\M$ is an invariant and reducing subspace for every $T \in \B(\H)$. They are called trivial subspaces.
In this paper, we will call $\M$ is an invariant or reducing subspace for $T$ only if $\M$ is a closed subset of $\H$.

\subsection{$C^*$-algebras}

For bounded linear operators in $\B(\H)$, we introduce the operator norm which will induce the definition of $C^{*}$-algebras in $\B(\H)$. Hereinafter, we will assume $\K=\C$.

\begin{definition}[Operator norm of $\B(\H)$]
Let $\norm{\cdot}:\B(\H) \to \R$ be defined by:
\begin{equation*}
\norm{T}=\sup\{\norm{Tx} \mid x \in H, \norm{x} \le 1 \}, \text{ where } T \in \B(\H).
\end{equation*}
Then $\norm{\cdot}$ is a norm on $\B(\H)$. We call $\norm{\cdot}$ the operator norm of $\B(\H)$. We will assume $\B(\H)$ has a topology induced by the operator norm $\norm{\cdot}$ such that a net $\{ T_\lambda \} \subseteq \B(\H)$ converges to $T \in B(\H)$ if and only if $\{ \norm{T-T_\lambda} \} \subseteq \R$ converges to $0$.
\end{definition}

For $S,T \in \B(\H)$, the sum and the product of those are defined as follows:
\begin{equation*}
(S+T)x=Sx+Tx, \text{ where } x \in \H,
\end{equation*}
\begin{equation*}
(ST)x=S(Tx), \text{ where } x \in \H.
\end{equation*}
Then $S+T$ and $ST \in \B(\H)$. By those definitions, we can regard $\B(\H)$ as a (non-commutative) ring.
For $\alpha \in \K$ and $T \in \B(\H)$, we also define the bounded linear operator $\alpha T$ as follows:
\begin{equation*}
(\alpha T)x=\alpha (Tx), \text{ where } x \in \H.
\end{equation*}

We define $C^{*}$-algebras in $\B(\H)$ as subrings of $\B(\H)$.
\begin{definition}[$C^*$-algebras]
Let $\A \subseteq \B(\H)$. We say that $\A$ is a $C^*$-algebra when the following conditions hold:
\begin{enumerate}[(i)]
\item If $S,T \in \A$, then $S+T \in \A$;
\item If $S,T \in \A$, then $ST \in \A$;
\item If $\alpha \in \K$ and $T \in \A$, then $\alpha T \in \A$;
\item If $T \in \A$, then $T^* \in \A$;
\item $\A$ is closed for the topology induced by the operator norm $\norm{\cdot}$ of $\B(\H)$.
\end{enumerate}
If $\A$ is a $C^*$-algebra and $I \in \A$, then we say that $\A$ is unital.
For $\A \subseteq \B(\H)$, the $C^*$-algebra generated by $\A$ is the closure of the following subset of $\B(\H)$:
\begin{equation*}
\left\{\sum_{i=1}^n \alpha_i T_i \mid n \in \N, \alpha_i \in \K, T_i \text{ is an element or a product of elements in } \A\right\}.
\end{equation*}
The $C^*$-algebra generated by $\A$ is denoted by $C^*(\A)$. For a finite set $\A=\{S_1,S_2,\ldots,S_n\}$ $(n \in \N)$, we will often write $C^*(S_1,S_2,\ldots,S_n)$ instead of $C^*(\A)$.
\end{definition}

Let us recall the notion of commutants and cyclic vectors and state fundamental results related to those notion without the proofs. For more detail, please refer to \cite{Co00GSM}.
\begin{definition}[Commutants]
For $\A \subseteq \B(\H)$, we define the commutant of $\A$, denoted by $\A'$, as follows:
\begin{equation*}
\A'=\{S \in \B(\H) \mid TS=ST(\forall T \in \A)\}.
\end{equation*}
And, we define the bicommutant of $\A$, denoted by $\A''$, as the commutant of $A'$.
\end{definition}

\begin{definition}[Cyclic vectors]
Let $\A \subseteq \B(\H)$. A cyclic vector for $\A$ is $x \in \H$ such that
\begin{equation*}
\overline{\A x}=\overline{\{Tx \mid T \in \A\}}=\H.
\end{equation*}
\end{definition}

\begin{proposition}[{\cite[Chapter 2, Proposition 14.3]{Co00GSM}}, Separating vector]
Let $\A \subseteq \B(\H)$ be a unital $C^*$-algebra and $x \in \H$. Then the following statements are equivalent:
\begin{enumerate}[(i)]
\item $x$ is a cyclic vector for $\A$; \label{cyclic}
\item $x$ is a separating vector for $\A'$, i.e., if $T \in \A'$ and $Tx=0$, then $T=0$. \label{separating}
\end{enumerate}
\end{proposition}

We will make use of the following theorem.
\begin{theorem}[{\cite[Chapter 5, Theorem 32.6]{Co00GSM}}, Irreducibility] \label{thm:Irreducibility}
If $\A \subseteq \B(\H)$ is a $C^*$-algebra, then the following statements are equivalent:
\begin{enumerate}[(i)]
\item $\A$ has no non-trivial reducing subspaces;
\item $\A'=\C I=\{ \alpha I \mid \alpha \in \C \}$;
\item If $x$ is any non-zero vector in $\H$, then $\H=\overline{\A x}$, that is, $x$ is a cyclic vector for $\A$.
\end{enumerate}
\end{theorem}

\section{Formulation of the Collatz conjecture by operators} \label{sec:Formulation of the Collatz conjecture}

In this section, $f$ will denote the Collatz map and $\H$ will denote a Hilbert space on $\C$ with $\dim{\H}=\aleph_0$. We will formulate some problems on operator theory which are arising from $f$.

\subsection{Formulation by a single operator}

Let $\{e_n\}_{n \in \N}$ be a C.O.N.S. for $\H$. We define $T:\H\to\H$ by:
\begin{equation*}
Te_n=e_{f(n)}, \text{ where }n \in \N,
\end{equation*}
and
\begin{equation*}
T\left(\sum_{n=1}^N x_n e_n\right)=\sum_{n=1}^N x_n Te_n, \text{ where } N \in \N, \text{ } x_n \in \C.
\end{equation*}
For $x=\sum_{n \in \N} x_n e_n$, we can compute the norm of $x$ as $\norm{x}=\sqrt{\sum_{n \in \N}\abs{x_n}^2}$. Notice that $f^{-1}(n)$ includes at most two elements for any $n \in \N$ and so $\abs{\sum_{m \in f^{-1}(n)}x_m}^2 \leq (\sum_{m \in f^{-1}(n)}\abs{x_m})^2 \leq 2\sum_{m \in f^{-1}(n)}\abs{x_m}^2$. Since
\begin{align*}
\sum_{n \in \N}\abs{\sum_{m \in f^{-1}(n)}x_m}^2
\leq\sum_{n \in \N}2\sum_{m \in f^{-1}(n)}\abs{x_m}^2
=2\sum_{n \in \N}\abs{x_n}^2
=2\norm{x}^2,
\end{align*}
there exists $\sum_{n \in \N}\left(\sum_{m \in f^{-1}(n)}x_m\right)e_n \in \H$. Thus, we can define $Tx$ as
\begin{equation*}
Tx=\sum_{n \in \N}x_n Te_n=\sum_{n \in \N}x_n e_{f(n)}=\sum_{n \in \N}\left(\sum_{m \in f^{-1}(n)}x_m\right)e_n,
\end{equation*}
and then, $T \in \B(\H)$.

In the setting above, we obtain the following relations.
\begin{lemma} \label{lem:Subset}
For every $n \in \N$,
\begin{equation*}
\overline{C^{*}(T)e_n} \subseteq \overline{\Span\{e_m \mid m \in \N, m \sim n\}}.
\end{equation*}
\end{lemma}
\begin{proof}
Let $\M=\overline{\Span\{e_m \mid m \in \N, m \sim n\}}$. We claim the following statement:
\begin{itemize}
\item If $e_l \in \M$, then $Te_l, T^{*}e_l \in \M$.
\end{itemize}
Once this is done, as $e_n \in \M$ and $\M$ is a closed linear subspace, we obtain $\overline{C^{*}(T)e_n} \subseteq \M$.

In order to prove the claim, pick $l \in \N$ such that $l \sim n$. Then, $e_l \in \M$ and $f(l) \sim n$. Thus, it holds that $Te_l=e_{f(l)} \in \M$.

To see $T^{*}e_l \in \M$, we use Theorem \ref{thm:Fourier series} and we obtain
\begin{align*}
T^{*}e_l=\sum_{k \in \N}\innpr{T^{*}e_l}{e_k}e_k
=\sum_{k \in \N}\innpr{e_l}{Te_k}e_k
=\sum_{k \in \N}\innpr{e_l}{e_{f(k)}}e_k
=\sum_{k \in f^{-1}(l)}e_k.
\end{align*}
When $k \in f^{-1}(l)$, it follows that $k \sim l$. This implies that $k \sim n$ and we obtain $e_k \in \M$.
Therefore, $T^{*}e_l=\sum_{k \in f^{-1}(l)}e_k \in \M$.
\end{proof}

According to the former lemma, we get an operator theoretic statement which is a sufficient condition for the Collatz conjecture.
\begin{theorem} \label{thm:Sufficient condition 1}
If there exists $n \in \N$ such that $e_n$ is a cyclic vector for $C^{*}(T)$, then the Collatz conjecture holds.
\end{theorem}
\begin{proof}
When $e_n$ is a cyclic vector for $C^{*}(T)$,
\begin{align*}
\H=\overline{C^{*}(T)e_n}\subseteq \overline{\Span\{e_m \mid m \in \N, m \sim n\}}
\end{align*}
by Lemma \ref{lem:Subset}.
It follows that $m \sim n$ for every $m \in \N$. Therefore, $1 \sim n$ and then, $1 \sim m$ for every $m \in \N$.
Since $1$ is a periodic point of $f$, $1 \in \orb(m;f)$. Thus, we can obtain the conclusion.
\end{proof}

The next corollary holds immediately.
\begin{corollary}[Sufficient condition]\label{cor:SingleS}
If $C^{*}(T)$ has no non-trivial reducing subspaces, then the Collatz conjecture holds.
\end{corollary}
\begin{proof}
From Theorem \ref{thm:Irreducibility}, if $C^{*}(T)$ has no non-trivial reducing subspaces, then every $x \in \H\backslash\{0\}$ is a cyclic vector for $C^{*}(T)$. Let $x=e_n$ for some $n \in \N$. Then the Collatz conjecture holds by Theorem \ref{thm:Sufficient condition 1}.
\end{proof}

\subsection{Formulation by two operators}

Let $T_1$, $T_2 \in \B(\H)$ be defined by:
\begin{align*}
T_1e_n&=
\begin{cases}
e_{3n+1}, & n:\text{odd},\\
0, & n:\text{even},
\end{cases}
\\
T_2e_n&=
\begin{cases}
0, & n:\text{odd},\\
e_{n/2}, & n:\text{even}.
\end{cases}
\end{align*}

It is not so hard to see the following statement.
\begin{lemma}
The operators $T_1$ and $T_2$ are partial isometries.
\end{lemma}
\begin{proof}
Let $\M_1=\overline{\Span\{e_n \mid n:\text{odd}\}}$ and $\M_2=\M_1^{\bot}=\overline{\Span\{e_n \mid n:\text{even}\}}$.

For every $x=\sum_{n:\text{odd}}\innpr{x}{e_n}e_n \in \M_1$,
\begin{align*}
\norm{x}^2&=\sum_{n:\text{odd}}\abs{\innpr{x}{e_n}}^2,\\
T_1x&=\sum_{n:\text{odd}}\innpr{x}{e_n}e_{3n+1},\\
T_2x&=0,
\end{align*} 
and
\begin{align*}
\norm{T_1x}^2=\sum_{n:\text{odd}}\abs{\innpr{x}{e_n}}^2=\norm{x}^2.
\end{align*}
Thus, $\norm{T_1x}=\norm{x}$.

For every $y=\sum_{n:\text{even}}\innpr{y}{e_n}e_n \in \M_2=\M_1^{\bot}$,
\begin{align*}
\norm{y}^2&=\sum_{n:\text{even}}\abs{\innpr{y}{e_n}}^2,\\
T_1y&=0,\\
T_2y&=\sum_{n:\text{even}}\innpr{y}{e_n}e_{n/2},\\
\end{align*}
and
\begin{align*}
\norm{T_2y}^2=\sum_{n:\text{even}}\abs{\innpr{y}{e_n}}^2=\norm{y}^2.
\end{align*}
Thus, $\norm{T_2y}=\norm{y}$.
Therefore, $T_1$ and $T_2$ are partial isometries.
\end{proof}

For $T_1$, $T_2$ defined as above, the following relations hold.
\begin{lemma} \label{lem:Equality 1}
For every $n \in \N$,
\begin{equation*}
\overline{C^{*}(T_1,T_2)e_n} = \overline{\Span\{e_m \mid m \in \N, m \sim n\}}.
\end{equation*}
\end{lemma}
\begin{proof}
The proof of $\subseteq$ is the same as Lemma \ref{lem:Subset}.

To prove the converse, it suffices to see that $e_m \in \overline{C^*(T_1,T_2)e_n}$ for every $m \in \N$ satisfying $m \sim n$.
Let us pick $m \in \N$ such that $m \sim n$. Then there exist two non-negative integers $k,l$ satisfying $f^k(m)=f^l(n)$. We can find a $k$-tuple $(i_1,i_2,\cdots,i_k) \in \{1,2\}^{*}$ and a $l$-tuple $(j_1,j_2,\cdots,j_l) \in \{1,2\}^{*}$ satisfying
\begin{align*}
(T_{i_1}T_{i_2} \cdots T_{i_k})e_m&=e_{f^k(m)},\\
(T_{j_1}T_{j_2} \cdots T_{j_l})e_n&=e_{f^l(n)}.
\end{align*}
It follows that
\begin{align*}
\innpr{(T_{i_1}T_{i_2} \cdots T_{i_k})^{*}e_{f^k(m)}}{e_m}
&=\innpr{e_{f^k(m)}}{(T_{i_1}T_{i_2} \cdots T_{i_k})e_m}\\
&=\innpr{e_{f^k(m)}}{e_{f^k(m)}}=1.
\end{align*}
Since $T_1$ and $T_2$ are partial isometries, $\norm{(T_{i_1}T_{i_2} \cdots T_{i_k})^{*}} \leq 1$. Thus,
\begin{align*}
1&=\abs{\innpr{(T_{i_1}T_{i_2} \cdots T_{i_k})^{*}e_{f^k(m)}}{e_m}}^2\\
&\leq\sum_{p \in \N}\abs{\innpr{(T_{i_1}T_{i_2} \cdots T_{i_k})^{*}e_{f^k(m)}}{e_p}}^2
=\norm{(T_{i_1}T_{i_2} \cdots T_{i_k})^{*}e_{f^k(m)}}^2 \leq 1.
\end{align*}
That is, $\innpr{(T_{i_1}T_{i_2} \cdots T_{i_k})^{*}e_{f^k(m)}}{e_p}=0$ if $m \neq p$.
Hence
\begin{align*}
(T_{i_1}T_{i_2} \cdots T_{i_k})^{*}e_{f^k(m)}
=\innpr{(T_{i_1}T_{i_2} \cdots T_{i_k})^{*}e_{f^k(m)}}{e_m}e_m
=e_m
\end{align*}
and
\begin{align*}
e_m&=(T_{i_1}T_{i_2} \cdots T_{i_k})^{*}e_{f^k(m)}
=(T_{i_1}T_{i_2} \cdots T_{i_k})^{*}e_{f^l(n)}\\
&=(T_{i_1}T_{i_2} \cdots T_{i_k})^{*}(T_{j_1}T_{j_2} \cdots T_{j_l})e_n.
\end{align*}
Therefore, $e_m \in \overline{C^{*}(T_1,T_2)e_n}$ and it completes the proof.
\end{proof}

By using the former lemma, we can obtain the following theorem and corollary.
\begin{theorem} \label{thm:Equivalent 1}
The following statements are equivalent:
\begin{enumerate}[(i)]
\item There exists $n \in \N$ such that $e_n$ is a cyclic vector for $C^{*}(T_1,T_2)$; \label{TwoOperators}
\item The Collatz conjecture holds. \label{TwoCollatz}
\end{enumerate}
\end{theorem}
\begin{proof}

$(\ref{TwoOperators}) \implies (\ref{TwoCollatz})$.
Suppose that there exists $n \in \N$ such that $e_n$ is a cyclic vector for $C^{*}(T_1,T_2)$.
By using Lemma \ref{lem:Equality 1}, it follows that
\begin{align*}
\H=\overline{C^{*}(T_1,T_2)e_n}=\overline{\Span\{e_m \mid m \in \N, m \sim n\}}.
\end{align*}
Then, $m \sim n$ for every $m \in \N$ and it implies that the Collatz conjecture holds.

$(\ref{TwoCollatz}) \implies (\ref{TwoOperators})$.
Assume that the Collatz conjecture holds. Then $1 \sim m$ for every $m \in \N$.
By using Lemma \ref{lem:Equality 1} again,
\begin{align*}
\overline{C^{*}(T_1,T_2)e_1}&=\overline{\Span\{e_m \mid m \in \N, m \sim 1\}}
=\overline{\Span\{e_m \mid m \in \N\}}=\H.
\end{align*}
Thus, $e_1$ is a cyclic vector for $C^{*}(T_1,T_2)$.

\end{proof}

\begin{corollary}[Sufficient condition] \label{cor:Sufficient condition 2}
If $C^{*}(T_1,T_2)$ has no non-trivial reducing subspaces, then the Collatz conjecture holds.
\end{corollary}
\begin{proof}
If $C^{*}(T_1,T_2)$ has no non-trivial reducing subspaces, then $e_n$ is a cyclic vector for every $n \in \N$. Then the Collatz conjecture holds by Theorem \ref{thm:Equivalent 1}.
\end{proof}

In the rest of this subsection, we will prove the necessity of the statement that $C^{*}(T_1,T_2)$ has no non-trivial reducing subspaces for the Collatz conjecture.
To begin with the proof of necessity, we will show the following lemma.
\begin{lemma} \label{lem:Projection 1}
If $Q \in C^{*}(T_1,T_2)'$ is a projection, then $Q e_1 \in \{0,e_1\}$.
\end{lemma}
\begin{proof}
We remark that $T_2^2T_1e_1=e_1$ and $e_1=T_1^*(T_2^*)^2$.

Let $Q \in C^{*}(T_1,T_2)'$ be a projection and assume $Q e_1 \neq 0$. We will show $Q e_1=e_1$.

If there exists $n \in \N\backslash\{1\}$ satisfying $\innpr{Q e_1}{e_n} \neq 0$, then we assume $n$ is the minimum integer as such. We have
\begin{align*}
0 \neq \innpr{Q e_1}{e_n}=\innpr{QT_1^*(T_2^*)^2e_1}{e_n}=\innpr{T_1^*(T_2^*)^2Qe_1}{e_n}=\innpr{Q e_1}{T_2^2T_1 e_n}.
\end{align*}
It follows that $T_2^2T_1 e_n \neq 0$. Therefore,
\begin{align*}
0 &\neq T_2^2T_1 e_n
=T_2^2 e_{3n+1}
=T_2 e_{(3n+1)/2}
=e_{(3n+1)/4}
\end{align*}
and
\begin{equation*}
0 \neq \innpr{Q e_1}{e_{(3n+1)/4}}.
\end{equation*}
But, $(3n+1)/4 < n$ where $n \in \N\backslash\{1\}$. This contradicts the minimality of $n$. Thus $\innpr{Q e_1}{e_n}=0$ for every $n \in \N\backslash\{1\}$.

Hence,
\begin{align*}
Q e_1&=\sum_{n \in \N}\innpr{Q e_1}{e_n}e_n
=\innpr{Q e_1}{e_1}e_1=\norm{Q e_1}^2e_1.
\end{align*}
Since $Q e_1 \neq 0$, it holds that $\norm{Q e_1}=1$ and $Q e_1=e_1$.
\end{proof}

Now, we can prove the next theorem.
\begin{theorem}[Necessary condition] \label{thm:Necessary condition 1}
If the Collatz conjecture holds, then $C^{*}(T_1,T_2)$ has no non-trivial reducing subspaces.
\end{theorem}
\begin{proof}
For every $C^*$-algebra $\A \subseteq \B(\H)$, $\A'$ is a von Neumann algebra, which is a special kind of $C^{*}$-algebras. Then $\A'$ is the norm closed linear span of its projection (see, e.g. \cite[Chapter 2, Proposition 13.3]{Co00GSM}). Thus,
\begin{align*}
&\text{$C^{*}(T_1,T_2)$ has no reducing subspaces}\\
&\iff C^{*}(T_1,T_2)'=\C I\\
&\iff \text{the set of all projections in $C^{*}(T_1,T_2)'$ is $\{0,I\}$}.
\end{align*}

Let us assume that the Collatz conjecture holds. Let $n \in \N$ and $Q \in C^{*}(T_1,T_2)'$ be a projection. Then, there exist $k \in \N$ such that $1=f^k(n)$ and a $k$-tuple $(i_1,i_2,\cdots,i_k) \in \{1,2\}^{*}$ satisfying
\begin{equation*}
(T_{i_1}T_{i_2} \cdots T_{i_k})e_n=e_1.
\end{equation*}
In the same way as in the proof of Lemma \ref{lem:Equality 1}, we have
\begin{equation*}
(T_{i_1}T_{i_2} \cdots T_{i_k})^{*}e_1=e_n.
\end{equation*}
By Lemma \ref{lem:Projection 1}, $Q e_1$ is either $0$ or $e_1$.

When $Q e_1=0$,
\begin{align*}
Q e_n&=Q(T_{i_1}T_{i_2} \cdots T_{i_k})^{*}e_1
=(T_{i_1}T_{i_2} \cdots T_{i_k})^{*}Q e_1=0.
\end{align*}
Thus, $Q e_n=0$ for every $n \in \N$. It follows that $Q=0$.

When $Q e_1=e_1$,
\begin{align*}
Q e_n&=Q(T_{i_1}T_{i_2} \cdots T_{i_k})^{*}e_1
=(T_{i_1}T_{i_2} \cdots T_{i_k})^{*}Q e_1
=(T_{i_1}T_{i_2} \cdots T_{i_k})^{*}e_1=e_n.
\end{align*}
Thus, $Q e_n=e_n$ for every $n \in \N$. It follows that $Q=I$.

Therefore, $Q \in \{0,I\}$ and it implies that $C^{*}(T_1,T_2)'=\C I$. Then, $C^{*}(T_1,T_2)$ has no non-trivial reducing subspaces.
\end{proof}

From Theorem \ref{thm:Necessary condition 1} and Corollary \ref{cor:Sufficient condition 2}, we obtain the following result.
\begin{theorem}[Necessary and sufficient condition]\label{thm:TwoNS}
The following are equivalent:
\begin{enumerate}[(i)]
\item $C^{*}(T_1,T_2)$ has no non-trivial reducing subspaces;
\item The Collatz conjecture holds.
\end{enumerate}
\end{theorem}

\subsection{Formulation by the Cuntz algebra} \label{subsection:FC}

In this subsection, we discuss formulation of the Collatz conjecture by the Cuntz algebras. 

While we defined C*-algebras as subalgebras of $\B(\H)$ for some Hilbert space H in Section 2, they can be characterized abstractly without using Hilbert spaces (see e.g., \cite{Co00GSM}, \cite{Co96GTM} for more details). The notions of projections, isometries, and partial isometries can be defined in the same way. The Cuntz algebras, which were introduced by Cuntz in \cite{Cu77CMP}, are typical examples of such abstract C*-algebras. They are defined as universal C*-algebras given by generators and relations as follows. 

\begin{definition}[The Cuntz algebras]
For $n \in \N_{>1}$, the Cuntz algebra $\O_n$ is the universal $C^{*}$-algebra generated by isometries $S_1,S_2, \cdots, S_n$ satisfying
\begin{enumerate}[(i)]
\item $\sum_{i=1}^n S_{i}S_{i}^{*}=I$,
\item $S_{i}^{*}S_{j}=\delta_{i,j}I, \text{ where } 1\leq i,j \leq n$.
\end{enumerate}
\end{definition}
Cuntz showed that $\O_n$ is simple as a $C^{*}$-algebra in \cite{Cu77CMP}, and so, for any isometries $S_1,S_2,\cdots,S_n \in \B(\H)$ satisfying the relation  above, it holds that $C^{*}(S_1,S_2,\cdots,S_n)$ is isomorphic to $\O_n$.

In this subsection, we formulate the Collatz conjecture as a problem of operator theory by using a $C^{*}$-algebra which is isomorphic to the Cuntz algebra $\O_2$.
Now we define two disjoint subsets of $\N$ by:
\begin{align*}
N_1&=\{n \in \N \mid n \equiv 1,5 \pmod{6}\},\\
N_2&=f(N_1)=\{n \in \N \mid n \equiv 4,16 \pmod{18}\}.
\end{align*}
In what follows, we will let $P$ denote the first-return map for $f$ on $N_1 \cup N_2$.

Since $N_2=f(N_1)$, $P(n)$ is defined for every $n \in N_1$, and $P(n)=3n+1$.
By definition of $N_2$, $3$ is not a prime factor of $n \in N_2$, and so there exists $m \in \N$ such that $f^{m}(n)=n/2^{m} \in N_1$. It follows that $P(n)$ is defined for every $n \in N_2$.
Furthermore, the following propositions hold.
\begin{proposition} \label{prop:First-return}
$P|_{N_1}, P|_{N_2}$ are injections and $P(N_1)=N_2, P(N_2)=N_1 \cup N_2$.
\end{proposition}
\begin{proof}
We will show that $P|_{N_2}$ is an injection and $P(N_2)=N_1 \cup N_2$.

For every $n \in N_2$, $n/P(n)$ is a power of $2$. 
Assume $P(m)=P(n)$ for $m,n\in N_2$ with $m>n$. 
As $m,n$ are products of $P(m)=P(n)$ and powers of $2$, $m$ is a multiple of $n$. 
It follows that $P(m)$ is not less than $n\in N_2$. 
So $P(n)=P(m)\geq n$, which is a contradiction. 
Therefore, $P|_{N_2}$ is a injection.

To prove the surjectivity of $P|_{N_2}$, we check the following statement:
\begin{itemize}
\item For every $n \in N_1 \cup N_2$, there exists a positive integer $k$ such that $2^{k} \cdot n \in N_2$.
\end{itemize}
Let $k$ be the minimum integer as such. Then, it holds that $P(2^{k} \cdot n)=n$ and $P|_{N_2}$ is surjective.

Notice that $N_1=\{n \in \N \mid n \equiv 1,5 \pmod{6}\}=\{n \in \N \mid n \equiv 1,5,7,11,13,17 \pmod{18}\}$. We prove the statement by division into cases:
\begin{itemize}
\item If $n \equiv 1,4,13 \pmod{18}$, then $4n \equiv 4,16 \pmod{18} \in N_2$.
\item If $n \equiv 5 \pmod{18}$, then $8n \equiv 4 \pmod{18} \in N_2$.
\item If $n \equiv 7,16 \pmod{18}$, then $16n \equiv 16,4 \pmod{18} \in N_2$.
\item If $n \equiv 11,17 \pmod{18}$, then $2n \equiv 4,16 \pmod{18} \in N_2$.
\end{itemize}
From the above, we can obtain the conclusion.
\end{proof}

\begin{proposition}[Transformation of the Collatz conjecture] \label{prop:Transformation}
The Collatz conjecture holds if and only if $1 \in \orb(n;P)$ for every $n \in N_1 \cup N_2$.
\end{proposition}
\begin{proof}
Assume that the Collatz conjecture holds. For every $n \in N_1 \cup N_2$,  $1 \in \orb(n;f)$ and $\orb(n;f)$ is a finite set.
Since $\orb(n;P) \subseteq \orb(n;f)$, it follows that $\orb(n;P)$ is finite. Then, there are integers $l<m \in \N$ such that
\begin{equation*}
P^l(n)=P^m(n).
\end{equation*}
Let $k=P^l(n)$. Since $l<m$, it holds that $k=P^{m-l}(k)$. For sufficiently large $N \in \N$, $f^{N}(n) \in \{1,2,4\}$. Thus,
\begin{align*}
k&=P^{m-l}(k)
=P^{2(m-l)}(k)
\cdots
=P^{N(m-l)}(k) \in \{1,2,4\}.
\end{align*}
Because $k$ is in $N_1 \cup N_2$, $k \neq 2$. Since $P(4)=1$, $1 \in \orb(n;P)$.

Let us prove the converse.
For every $n \in \N$, there is a positive integer $m \in \N$ such that $f^m(n) \in N_1$. By assumption, $1 \in \orb(f^m(n);P)$ and  $\orb(f^m(n);P) \subseteq \orb(n;f)$ implies that the Collatz conjecture holds.
\end{proof}

Consider the Hilbert space $\H$ whose C.O.N.S. is $\{e_n\}_{n \in N_1 \cup N_2}$. We define $T_1,T_2 \in \B(\H)$ by:
\begin{align*}
T_1e_n&=
\begin{cases}
e_{P(n)}, & n \in N_1,\\
0, & n \in N_2,
\end{cases}
\\
T_2e_n&=
\begin{cases}
0, & n \in N_1,\\
e_{P(n)}, & n \in N_2.
\end{cases}
\end{align*}
Then we define $S_1,S_2 \in \B(\H)$ as $S_1=T_1^{*}T_2^{*}, S_2=T_2^{*}$.

\begin{proposition}
If $T_1,T_2,S_1$, and $S_2$ are defined as above, then $C^{*}(S_1,S_2)=C^{*}(T_1,T_2)$ and $C^{*}(S_1,S_2)$ is isomorphic to $\O_2$.
\end{proposition}
\begin{proof}
Since $S_1=T_1^{*}T_2^{*}, S_2=T_2^{*}$, $C^{*}(S_1,S_2) \subseteq C^{*}(T_1,T_2)$.
By Proposition \ref{prop:First-return}, $P|_{N_2}$ is injective and $P(N_1)=N_2$. Therefore, $T_2^{*}T_2T_1=T_1$.
Then, $T_1=T_2^{*}T_2T_1=S_2S_1^{*}, T_2=S_2^{*}$, and it follows that $C^{*}(T_1,T_2) \subseteq C^{*}(S_1,S_2)$.
Then, $C^{*}(S_1,S_2)=C^{*}(T_1,T_2)$.

Since $P(N_1)=N_2, P(N_2)=N_1 \cup N_2$, both $S_1=T_1^{*}T_2^{*}$ and $S_2=T_2^{*}$ are isometries. Furthermore,
\begin{align*}
S_1S_1^{*}e_n&=
\begin{cases}
e_n, & n \in N_1,\\
0, & n \in N_2,
\end{cases}
\\
S_2S_2^{*}e_n&=
\begin{cases}
0, & n \in N_1,\\
e_n, & n \in N_2.
\end{cases}
\end{align*}
Then, it holds that $S_1S_1^{*}+S_2S_2^{*}=I$.

Since $T_1T_2^{*}=0$ and $T_2T_1^{*}=0$, it follows that $S_1^{*}S_2=T_2T_1T_2^{*}=0=T_2T_1^{*}T_2^{*}=S_2^{*}S_1$.
Hence, $C^{*}(S_1,S_2)$ is isomorphic to $\O_2$.
\end{proof}

From now on, we will fix $T_1,T_2,S_1$, and $S_2$ as above and let $\sim$ denote the equivalence relation on $N_1 \cup N_2$ by $P$. This setting presents the following relations.

\begin{lemma} \label{lem:Equality 2}
For every $n \in N_1 \cup N_2$,
\begin{equation*}
\overline{C^{*}(S_1,S_2)e_n}=\overline{\Span\{e_m \mid m \in N_1 \cup N_2, m \sim n\}}.
\end{equation*}
\end{lemma}
\begin{proof}
It suffices to show that
\begin{align*}
\overline{C^{*}(T_1,T_2)e_n}=\overline{\Span\{e_m \mid m \in N_1 \cup N_2, m \sim n\}}.
\end{align*}
The way to prove is the same as the proof of Lemma \ref{lem:Equality 1}.
\end{proof}

Now, we have the next theorem which induces a sufficient condition for the Collatz conjecture as an operator theoretic statement.
\begin{theorem} \label{thm:Equivalent 2}
The following statements are equivalent:
\begin{enumerate}[(i)]
\item There exists $n \in N_1 \cup N_2$ such that $e_n$ is a cyclic vector for $C^{*}(S_1,S_2)$; \label{CuntzOperators}
\item The Collatz conjecture holds. \label{CuntzCollatz}
\end{enumerate}
\end{theorem}
\begin{proof}
Notice that $1 \in \orb(n;P)$ for every $n \in N_1 \cup N_2$ if and only if $1 \sim n$ for every $n \in N_1 \cup N_2$.

$(\ref{CuntzOperators}) \implies (\ref{CuntzCollatz})$.
Assume that there exists $n \in N_1 \cup N_2$ such that $e_n$ is a cyclic vector for $C^{*}(S_1,S_2)$.
By Lemma \ref{lem:Equality 2}, it follows that
\begin{align*}
\H=\overline{C^{*}(S_1,S_2)e_n}
=\overline{\Span\{e_m \mid m \in N_1 \cup N_2, m \sim n\}}.
\end{align*}
Hence, $m \sim n$ for every $m \in N_1 \cup N_2$ which implies that $1 \sim m$ for every $m \in N_1 \cup N_2$. Then, the Collatz conjecture holds by Proposition \ref{prop:Transformation}.

$(\ref{CuntzCollatz}) \implies (\ref{CuntzOperators})$.
Assume that the Collatz conjecture holds. Then $1 \sim m$ for every $m \in \N$.
By Proposition \ref{prop:Transformation} and Lemma \ref{lem:Equality 2}, it follows that
\begin{align*}
\overline{C^{*}(S_1,S_2)e_1}&=\overline{\Span\{e_m \mid m \in N_1 \cup N_2, m \sim 1\}}\\
&=\overline{\Span\{e_m \mid m \in N_1 \cup N_2\}}
=\H.
\end{align*}
Thus, $e_1$ is a cyclic vector for $C^{*}(S_1,S_2)$.
\end{proof}

\begin{corollary}[Sufficient condition] \label{cor:Sufficient condition 3}
If $C^{*}(S_1,S_2)$ has no non-trivial reducing subspaces, then the Collatz conjecture holds.
\end{corollary}
\begin{proof}
If $C^{*}(S_1,S_2)$ has no non-trivial reducing subspaces, then $e_n$ is a cyclic vector for every $n \in N_1 \cup N_2$. Then the Collatz conjecture holds by Theorem \ref{thm:Equivalent 2}.
\end{proof}

In the rest of this subsection, we will prove the necessity of the statement that $C^{*}(S_1,S_2)$ has no non-trivial reducing subspaces for the Collatz conjecture.
It will arise from the following lemma.

\begin{lemma} \label{lem:Projection 2}
If $Q \in C^{*}(S_1,S_2)'$ is a projection and the Collatz conjecture holds, then $Q e_1 \in \{0,e_1\}$.
\end{lemma}
\begin{proof}
We remark that $T_2T_1e_1=e_1$ and $e_1=T_1^*T_2^*e_1$.

Let $Q \in C^{*}(S_1,S_2)'=C^{*}(T_1,T_2)'$ be a projection and assume $Qe_1 \neq 0$. We will show $Qe_1=e_1$.

If there exists $n \in (N_1 \cup N_2)\backslash\{1\}$ satisfying $\innpr{Q e_1}{e_n} \neq 0$, then
\begin{align*}
0 \neq \innpr{Q e_1}{e_n}=\innpr{QT_1^*T_2^*e_1}{e_n}=\innpr{T_1^*T_2^*Qe_1}{e_n}=\innpr{Qe_1}{T_2T_1e_n}.
\end{align*}
Similarly, for every $k \in \N$,
\begin{align*}
0 \neq \innpr{Q e_1}{(T_2T_1)^{k}e_n}.
\end{align*}
Then, $(T_2T_1)^{k}e_n \notin \ker{T_1}$. Therefore,
\begin{equation*}
(T_2T_1)^{k}e_n=e_{P^{2k}(n)}
\end{equation*}
and
\begin{equation*}
P^{2k}(n) \in N_1.
\end{equation*}
If the Collatz conjecture holds, then $1 \in \orb(n;P)$ by Proposition \ref{prop:Transformation}. It follows that $P^{2k}(n) \in \{1,4\}$ for sufficiently large $k \in \N$. Since $P^{2k}(n) \in N_1$, $P^{2k}(n)=1$.
Then,
\begin{equation*}
(T_2T_1)^{k}e_n=e_1.
\end{equation*}
which implies that
\begin{align*}
\innpr{e_1}{e_n}&=\innpr{(T_1^{*}T_2^{*})^{k}e_1}{e_n}
=\innpr{e_1}{(T_2T_1)^{k}e_n}
=\innpr{e_1}{e_1}=1.
\end{align*}
But, this contradicts the assumption $n \in (N_1 \cup N_2)\backslash\{1\}$.
Thus, $\innpr{Q e_1}{e_n}=0$ for every $n \in (N_1 \cup N_2)\backslash\{1\}$.

Hence,
\begin{align*}
Q e_1=\sum_{n \in N_1 \cup N_2}\innpr{Q e_1}{e_n}e_n=\innpr{Q e_1}{e_1}e_1=\norm{Q e_1}^2 e_1.
\end{align*}
Since $Q e_1 \neq 0$, it holds that $\norm{Q e_1}=1$ and $Q e_1=e_1$.
\end{proof}

\begin{theorem}[Necessary condition] \label{thm:Necessary condition 2}
If the Collatz conjecture holds, then $C^{*}(S_1,S_2)$ has no non-trivial reducing subspaces.
\end{theorem}
\begin{proof}
By using Lemma \ref{lem:Equality 2} and \ref{lem:Projection 2}, the proof is the same as Theorem \ref{thm:Necessary condition 1}.
\end{proof}

From Theorem \ref{thm:Necessary condition 2} and Corollary \ref{cor:Sufficient condition 3}, we obtain the next result.
\begin{theorem}[Necessary and sufficient condition]\label{thm:CuntzNS}
The following are equivalent:
\begin{enumerate}[(i)]
\item $C^{*}(S_1,S_2)$ has no non-trivial reducing subspaces;
\item The Collatz conjecture holds.
\end{enumerate}
\end{theorem}

\section{Generalization to dynamical systems with bounded condition} \label{sec:Generalization}

In this section, we will generalize the Collatz conjecture and formulations as operator theory. Assume that $X$ is an arbitrary set whose cardinality is $\aleph_0$ and $f:X \to X$ is a map.

\subsection{Bounded condition}

First, we introduce the notion of the bounded condition for $f$ as below.
\begin{definition}[Bounded condition]
When there exist disjoint subsets $X_1,\\X_2, \cdots ,X_k$ of $X$, where $k \in \N$ such that:
\begin{enumerate}[(i)]
\item $\bigcup_{i=1}^k X_i=X$;
\item $f|_{X_i}$ is injective, where $1 \leq i \leq k$,
\end{enumerate}
we say that $f$ satisfies the bounded condition.
\end{definition}

\begin{example}
\begin{enumerate}[(i)]
\item ($qx{+}1$-function)
Let $q>0$ be an odd integer. The $qx{+}1$-function $f_q:\N \to \N$ is defined by:
\begin{equation*}
f_q(n)=
\begin{cases}
qn+1, & n:\text{odd},\\
n/2, & n:\text{even}.
\end{cases}
\end{equation*}
The restriction of $f_q$ to the set of all positive odd integers is injective. It is also true for the set of all positive even integers. Thus, $f_q$ satisfies the bounded condition.

\item ($3x{+}d$-function)
Let $d>0$ be an odd integer. The $3x{+}d$-function $f_d:\N \to \N$ is defined by:
\begin{equation*}
f_d(n)=
\begin{cases}
3n+d, & n:\text{odd},\\
n/2, & n:\text{even}.
\end{cases}
\end{equation*}
The restriction of $f_d$ to the set of all positive odd integers is injective. It is also true for the set of all positive even integers. Thus, $f_d$ satisfies the bounded condition.

\end{enumerate}
\end{example}

Hereinafter, we will assume that $f$ satisfies the bounded condition.
Now, we generalize the Collatz conjecture as the following problem:
\begin{problem}
For given $X$ and $f$, does it hold that $x \sim y$ for every $x,y \in X$?
\end{problem}

It is easy to see that this problem does not have an affirmative answer for every $f$ satisfying the bounded condition. When $f=\id_X$, $f$ satisfies the bounded condition and $x \sim y$ implies $x=y$. This gives a counterexample.
Thus, our interests are when the statement of the problem is true for $f$ and how we can verify the statement is true or false for $f$.

In order to formulate this problem by operator theory, let $\H$ be a Hilbert space on $\C$ with $\dim\H=\aleph_0$ and $\{e_x\}_{x \in X}$ be a C.O.N.S. for $\H$.
We define $T:\H\to\H$ by:
\begin{equation*}
Te_x=e_{f(x)}, \text{ where }x \in X,
\end{equation*}
and
\begin{equation*}
T\left(\sum_{x \in F} a_x e_x\right)=\sum_{x \in F} a_x Te_x, \text{ where $F$ is a finite subset of $X$, } a_x \in \C.
\end{equation*}
According to the bounded condition for $f$, $f^{-1}(x)$ includes at most $k$ elements for any $x \in X$ and so
\begin{align*}
\abs{\sum_{y \in f^{-1}(x)}a_y}^2 \leq (\sum_{y \in f^{-1}(x)}\abs{a_y})^2 \leq k\sum_{y \in f^{-1}(x)}\abs{a_y}^2.
\end{align*}
Since
\begin{align*}
\sum_{x \in X}\abs{\sum_{y \in f^{-1}(x)}a_y}^2
\leq\sum_{x \in X}k\sum_{y \in f^{-1}(x)}\abs{a_y}^2
=k\sum_{x \in X}\abs{a_x}^2,
\end{align*}
there exists $\sum_{x \in X}\left(\sum_{y \in f^{-1}(x)}a_x\right)e_x \in \H$ for every $a=\sum_{x \in X}a_x e_x \in \H$. Thus, we can define $Ta$ as
\begin{equation*}
Ta=\sum_{x \in X}a_x Te_x=\sum_{x \in X}a_x e_{f(x)}=\sum_{x \in X}\left(\sum_{y \in f^{-1}(x)}a_y\right)e_x,
\end{equation*}
and then, $T \in \B(\H)$.

In the setting above, we obtain the following relations.
\begin{lemma} \label{lem:Subset 3}
For every $x \in X$,
\begin{equation*}
\overline{C^{*}(T)e_x} \subseteq \overline{\Span\{e_y \mid y \in X, y \sim x\}}.
\end{equation*}
\end{lemma}
\begin{proof}
Let $\M=\overline{\Span\{e_y \mid y \in X, x \sim y\}}$. We claim the following statement:
\begin{itemize}
\item If $e_z \in \M$, then $Te_z, T^{*}e_z \in \M$.
\end{itemize}
Once this is done, as $e_x \in \M$ and $\M$ is a closed linear subspace, we obtain $\overline{C^{*}(T)e_x} \subseteq \M$.

To prove the claim, pick $z \in X$ such that $z \sim x$. Since $e_z \in \M$ and $z \sim f(z)$, it holds that $Te_z=e_{f(z)} \in \M$.

In order to see $T^{*}e_z \in M$, we use Theorem \ref{thm:Fourier series} and we obtain
\begin{align*}
T^{*}e_z&=\sum_{w \in X}\innpr{T^{*}e_z}{e_w}e_w
=\sum_{w \in X}\innpr{e_z}{Te_w}e_w\\
&=\sum_{w \in X}\innpr{e_z}{e_{f(w)}}e_w
=\sum_{w \in f^{-1}(z)}e_w.
\end{align*}
When $w \in f^{-1}(z)$, it holds that $w \sim z$. This implies that $w \sim x$ and we obtain $e_w \in \M$.
Therefore, $T^{*}e_z=\sum_{w \in f^{-1}(z)}e_w \in \M$.
\end{proof}

\begin{theorem} \label{thm:Sufficient condition 4}
If there exists $x \in X$ such that $e_x$ is a cyclic vector for $C^{*}(T)$, then $y \sim z$ for every $y,z \in X$.
\end{theorem}
\begin{proof}
When $e_x$ is a cyclic vector for $C^{*}(T)$,
\begin{align*}
\H=\overline{C^{*}(T)e_x}\subseteq\overline{\Span\{e_y \mid y \in X, y \sim x\}}
\end{align*}
by Lemma \ref{lem:Subset 3}.
It follows that $y \sim x$ for every $y \in X$. Pick any $y,z \in X$. Since $y \sim x$ and $z \sim x$, it holds that $y \sim z$.
\end{proof}

\begin{corollary}
If $C^{*}(T)$ has no non-trivial reducing subspaces, then $x \sim y$ for every $x,y \in X$.
\end{corollary}
\begin{proof}
From Theorem \ref{thm:Irreducibility}, if $C^{*}(T)$ has no non-trivial reducing subspaces, then every $a \in \H\backslash\{0\}$ is a cyclic vector for $C^{*}(T)$. Let $a=e_x$ for some $x \in X$. Then $x \sim y$ for every $x,y \in X$ by Theorem \ref{thm:Sufficient condition 4}.
\end{proof}

Next, we define $k$ partial isometries $T_1,T_2,\cdots,T_k \in \B(\H)$ by:
\begin{align*}
T_i e_x=
\begin{cases}
e_{f(x)}, & x \in X_i,\\
0, & x \notin X_i,
\end{cases}
\end{align*}
where $1 \leq i \leq k$. This setting induces the following relations.

\begin{lemma} \label{lem:Equality 4}
For every $x \in X$,
\begin{equation*}
\overline{C^{*}(T_1,T_2,\cdots,T_k)e_x}=\overline{\Span\{e_y \mid y \in X, y \sim x\}}.
\end{equation*}
\end{lemma}
\begin{proof}
The proof of $\subseteq$ is the same as Lemma \ref{lem:Subset 3}.

In order to prove the converse, it suffices to see $e_y \in \overline{C^*(T_1,T_2)e_x}$ for every $y \in X$ such that $y \sim x$.
Let us pick $y \in X$ satisfying $y \sim x$. There exist two non-negative integers $m,n$ such that $f^m(y)=f^n(x)$. We can find an $m$-tuple $(i_1,i_2,\cdots,i_m) \in \{1,2,\cdots,k\}^{*}$ and an $n$-tuple $(j_1,j_2,\cdots,j_n) \in \{1,2,\cdots,k\}^{*}$ satisfying
\begin{align*}
(T_{i_1}T_{i_2} \cdots T_{i_m})e_y&=e_{f^m(y)},\\
(T_{j_1}T_{j_2} \cdots T_{j_n})e_x&=e_{f^n(x)}.
\end{align*}
It follows that
\begin{align*}
\innpr{(T_{i_1}T_{i_2} \cdots T_{i_m})^{*}e_{f^m(y)}}{e_y}
&=\innpr{e_{f^m(y)}}{(T_{i_1}T_{i_2} \cdots T_{i_m})e_y}\\
&=\innpr{e_{f^m(y)}}{e_{f^m(y)}}=1.
\end{align*}
Since $T_1,T_2,\cdots,T_k$ are partial isometries, it holds that $\norm{(T_{i_1}T_{i_2} \cdots T_{i_m})^{*}} \leq 1$. Thus,
\begin{align*}
1&=\abs{\innpr{(T_{i_1}T_{i_2} \cdots T_{i_m})^{*}e_{f^m(y)}}{e_y}}^2\\
&\leq\sum_{z \in X}\abs{\innpr{(T_{i_1}T_{i_2} \cdots T_{i_m})^{*}e_{f^m(y)}}{e_z}}^2
=\norm{(T_{i_1}T_{i_2} \cdots T_{i_m})^{*}e_{f^m(y)}}^2 \leq 1.
\end{align*}
That is, $\innpr{(T_{i_1}T_{i_2} \cdots T_{i_m})^{*}e_{f^m(y)}}{e_z}=0$ if $y \neq z$.
Hence
\begin{equation*}
(T_{i_1}T_{i_2} \cdots T_{i_m})^{*}e_{f^m(y)}=e_y
\end{equation*}
and
\begin{align*}
e_y&=(T_{i_1}T_{i_2} \cdots T_{i_m})^{*}e_{f^m(y)}
=(T_{i_1}T_{i_2} \cdots T_{i_m})^{*}e_{f^n(x)}\\
&=(T_{i_1}T_{i_2} \cdots T_{i_m})^{*}(T_{j_1}T_{j_2} \cdots T_{j_n})e_x.
\end{align*}
Therefore, it holds that $e_y \in \overline{C^{*}(T_1,T_2,\cdots,T_k)e_x}$, which completes the proof.
\end{proof}

From the former lemma, we obtain the following theorem and corollary.
\begin{theorem}[Necessary and sufficient condition] \label{thm:Equivalent 4}
The following statements are equivalent:
\begin{enumerate}[(i)]
\item There exists $x \in X$ such that $e_x$ is a cyclic vector for $C^{*}(T_1,T_2,\cdots,T_k)$; \label{kOperators}
\item $x \sim y$ for every $x,y \in X$. \label{kCollatz}
\end{enumerate}
\end{theorem}
\begin{proof}
$(\ref{kOperators}) \implies (\ref{kCollatz})$.
Suppose that there exists $x \in X$ such that $e_x$ is a cyclic vector for $C^{*}(T_1,T_2,\cdots,T_k)$. By using Lemma \ref{lem:Equality 4}, it follows that
\begin{align*}
\H=\overline{C^{*}(T_1,T_2,\cdots,T_k)e_x}
=\overline{\Span\{e_y \mid y \in X, y \sim x\}}.
\end{align*}
Then $y \sim x$ for every $y \in X$, which implies that $y \sim z$ for every $y,z \in X$.

$(\ref{kCollatz}) \implies (\ref{kOperators})$.
Assume $x \sim y$ for every $x,y \in X$ and fix $x \in X$. Then $x \sim y$ for every $y \in X$.
By using Lemma \ref{lem:Equality 4} again,
\begin{align*}
\overline{C^{*}(T_1,T_2,\cdots,T_k)e_x}&=\overline{\Span\{e_y \mid y \in X, y \sim x\}}
=\overline{\Span\{e_y \mid y \in X\}}
=\H.
\end{align*}
Thus, $e_x$ is a cyclic vector for $C^{*}(T_1,T_2,\cdots,T_k)$.
\end{proof}

\begin{corollary}[Sufficient condition]
If $C^{*}(T_1,T_2,\cdots,T_k)$ has no non-trivial reducing subspaces, then $x \sim y$ for every $x,y \in X$.
\end{corollary}
\begin{proof}
If $C^{*}(T_1,T_2,\cdots,T_k)$ has no non-trivial reducing subspaces, then $e_x$ is a cyclic vector for every $x \in X$. Then $x \sim y$ for every $x,y \in X$ by Theorem \ref{thm:Equivalent 4}.
\end{proof}

\subsection{Separating condition}

Secondly, we define the separating condition for $f$.
Let $n \in \N$ and $\sigma$ be a cyclic permutation of $\{1,2,\cdots,n\}$ defined by:
\begin{equation*}
1 \mapsto 2 \mapsto 3 \mapsto \cdots \mapsto n-1 \mapsto n \mapsto 1.
\end{equation*}
For an $n$-tuple of numbers $(i_1,i_2,\cdots,i_n)$, we say that it is aperiodic if
\begin{equation*}
(i_{\sigma^j(1)},i_{\sigma^j(2)},\cdots,i_{\sigma^j(n)})\neq (i_1,i_2,\cdots,i_n)
\end{equation*}
for every $1 \leq j <n$.

\begin{definition}[Separating condition]
Let $x \in X$. When the following conditions hold:
\begin{enumerate}[(i)]
\item There exists $n \in \N$ satisfying $f^n(x)=x$;
\item For the minimum $n$ as above, the $n$-tuple $(i_1,i_2,\cdots,i_n) \in \{1,2,\cdots,k\}^{*}$, such that $f^{j-1}(x) \in X_{i_j}$ where $1 \leq j \leq n$, is aperiodic,
\end{enumerate}
we say that $f$ satisfies the separating condition for $x$.
\end{definition}

For example, the Collatz map satisfies the separating condition for $1$. This holds as follows: $f^3(1)=1$ and $1 \in X_1, 2,4 \in X_2$ where $X_1,X_2$ are the sets of all positive odd numbers and even numbers. Since $3$-tuple $(1,2,2)$ is aperiodic, we obtain the conclusion.

\begin{example}
\begin{enumerate}[(i)]
\item ($qx{+}1$-function for $q=5$, Mersenne numbers)
Let $X_1$ and $X_2$ be the sets of all positive odd numbers and even numbers.
For $q=5$, $f_5$ makes a loop as follows:
\begin{equation*}
1 \mapsto 6 \mapsto 3 \mapsto 16 \mapsto 8 \mapsto 4 \mapsto 2 \mapsto 1.
\end{equation*}
Since $7$-tuple $(1,2,1,2,2,2,2)$ is aperiodic, $f_5$ satisfies the separating condition for $1$.
Let $q>0$ be a Mersenne number, i.e., $q=2^k-1$ for some $k \in \N$. Then, $f_q(1)=2^k, f_q^2(1)=2^{k-1}, \cdots, f_q^{k+1}(1)=1$ and $k+1$-tuple $(1,2,\cdots,2)$ is aperiodic. Hence, $f_q$ satisfies the separating condition for $1$.

\item ($3x{+}d$-function)
Let $d>0$ be an odd integer. Let $X_1=\{n \in \N \mid n:odd\}$ and $X_2=\{n \in \N \mid n:even\}$. Then, $f_d(d)=4d, f_d^2(d)=2d, f_d^3(d)=d$ and $3$-tuple $(1,2,2)$ is aperiodic. Hence, $f_d$ satisfies the separating condition for $d$.

\end{enumerate}
\end{example}

For a map which satisfies separating condition, the following holds.
\begin{lemma} \label{lem:Generating}
Let $x \in X$, $\M=\overline{C^{*}(T_1,T_2,\cdots,T_k)e_x}$, and assume $f$ satisfies the separating condition for $x \in X$.
Then, $\overline{C^{*}(T_1,T_2,\cdots,T_k)a}=\M$ for every $a \in \M\backslash\{0\}$.
\end{lemma}
\begin{proof}
From the assumption of the separating condition, there exists $n \in \N$ such that $f^n(x)=x$ and the $n$-tuple $(i_1,i_2,\cdots,i_n) \in \{1,2,\cdots,k\}^{*}$, such that $f^{j-1}(x) \in X_{i_j}$ where $1 \leq j \leq n$, is aperiodic.
Set $T_I=T_{i_n}T_{i_{n-1}} \cdots T_1$. Then
\begin{equation*}
T_I e_x=e_x
\end{equation*}
and it follows that
\begin{equation*}
T_I^{*}e_x=e_x.
\end{equation*}
By the assumption of aperiodicity, it holds that $T_I e_{f^{j}(x)}=0$ for $1 \leq j < n$.
Let $y \in X$. When $y \sim x$, it holds that $x \in \orb(y;f)$. Thus, for sufficiently large $m \in \N$, $T_I^m e_y \in \{0,e_x\}$. Assume $T_I^m e_y=e_x$.  Then,
\begin{align*}
e_y&=(T_I^{*})^{m}e_x=e_x,
\end{align*}
and it follows that $y=x$. Thus, for every $y \in X$ satisfying $y \sim x$ and $y \neq x$, it holds that $T_I^m e_y=0$ for sufficiently large $m \in \N$.

Let $a \in \M\backslash\{0\}$.
According to Lemma \ref{lem:Equality 4},
\begin{align*}
\M=\overline{C^{*}(T_1,T_2,\cdots,T_k)e_x}=\overline{\Span\{e_y \mid y \in X, y \sim x\}}.
\end{align*}
Hence, there exists $y \in X$ such that $y \sim x$ and $\innpr{a}{e_y} \neq 0$. Then, there exist $m \in \N$ and an $m$-tuple $(j_1,j_2,\cdots,j_m) \in \{1,2,\cdots,k\}^{*}$ satisfying
\begin{equation*}
T_{j_m}T_{j_{m-1}} \cdots T_{j_1}e_y=e_x.
\end{equation*}
Let $b=T_{j_m}T_{j_{m-1}} \cdots T_{j_1}a$. It follows that $\innpr{b}{e_x} \neq 0$, $b \in \overline{C^{*}(T_1,T_2,\cdots,T_k)a}$.
Since $\norm{T_I}\leq 1$,
\begin{equation*}
\norm{T_I^{m}b} \geq \norm{T_I^{m+1}b}
\end{equation*}
for every $m \in \N$. By the discussion about $T_I$ above, $T_I^m b \to \innpr{b}{e_x}e_x$ $(m \to \infty)$.
Then, $e_x \in \overline{C^{*}(T_1,T_2,\cdots,T_k)a}$ and it follows that
\begin{equation*}
\M=\overline{C^{*}(T_1,T_2,\cdots,T_k)e_x} \subseteq \overline{C^{*}(T_1,T_2,\cdots,T_k)a} \subseteq \M.
\end{equation*}
Hence, we obtain the conclusion.
\end{proof}

Separating condition of $f$ implies the next necessary condition.
\begin{theorem}[Necessary condition]
Let $x \in X$ and assume $f$ satisfies the separating condition for $x \in X$.
If $y \sim z$ for every $y,z \in X$, then $C^{*}(T_1,T_2,\cdots,T_k)$ has no non-trivial reducing subspaces.
\end{theorem}
\begin{proof}
Assume that $y \sim z$ for every $y,z \in X$. By Lemma \ref{lem:Equality 4}, it follows that
\begin{equation*}
\overline{C^{*}(T_1,T_2,\cdots,T_k)e_x}
=\overline{\Span\{e_y \mid y \in X, y \sim x\}}
=\overline{\Span\{e_y \mid y \in X\}}
=\H.
\end{equation*}
From Lemma \ref{lem:Generating}, we obtain that
\begin{equation*}
\overline{C^{*}(T_1,T_2,\cdots,T_k)a}=\H
\end{equation*}
for every $a \in \H\backslash\{0\}$. Then $a$ is a cyclic vector for $C^{*}(T_1,T_2,\cdots,T_k)$ and this implies that $C^{*}(T_1,T_2,\cdots,T_k)$ has no non-trivial reducing subspaces.
\end{proof}

\subsection{Cuntz-Krieger condition}

Thirdly, we introduce the notion of the Cuntz-Krieger condition for $f:X\to X$. This idea comes from a special class of $C^{*}$-algebras given by Cuntz and Krieger in \cite{CK80IM}. 
\begin{definition}[The Cuntz-Krieger algebras]
For $k \in \N_{>1}$ and $k \times k$ matrix $A=(A(i,j))$ with entries in $\{0,1\}$ such that no row is zero, the Cuntz-Krieger algebra $\O_{A}$ is the universal $C^{*}$-algebra generated by partial isometries $S_1,S_2,\cdots,S_k$ satisfying
\begin{enumerate}[(i)]
\item $\sum_{i=1}^n S_iS_i^{*}=I$;
\item $S_j^{*}S_j=\sum_{i=1}^n A(j,i)S_iS_i^{*}$.
\end{enumerate}
\end{definition}
Cuntz and Krieger proved that $\O_{A}$ is a simple $C^{*}$-algebra if $A$ is irreducible and not a permutation matrix in \cite{CK80IM}, and so, for any partial isometries $S_1,S_2,\cdots,S_k \in \B(\H)$ satisfying the relation above, $C^{*}(S_1,S_2,\cdots,S_k)$ is isomorphic to $\O_{A}$. The Cuntz algebras are special class of the Cuntz-Krieger algebras.

Keeping this concept in mind, we make the following definition. 
\begin{definition}[Cuntz-Krieger condition]
When the following hold:
\begin{enumerate}[(i)]
\item $f$ is surjective; \label{def:CuntzKriegerCondition1}
\item For every $1 \leq j \leq k$, there exists $J \subseteq \{1,2,\cdots,k\}$ such that $f(X_j)=\bigcup_{i \in J}X_i$, \label{def:CuntzKriegerCondition2}
\end{enumerate}
we say that $f$ satisfies the Cuntz-Krieger condition.
\end{definition}

For the Cuntz-Krieger condition of $f$, the following proposition holds.
\begin{proposition}
If $f$ satisfies the Cuntz-Krieger condition, then there exist the Cuntz-Krieger algebra $\O_A$ and a surjective $*$-homomorphism $\phi:\O_A \to C^{*}(T_1,T_2,\cdots,T_k)$.
\end{proposition}
\begin{proof}
Define $S_i=T_i^{*}$ for $1 \leq i \leq k$.
Then, $S_i$ is a partial isometry and $S_iS_i^{*}=T_i^{*}T_i$ is the projection for $\overline{\Span\{e_x|x \in X_i\}}$.
Hence, $\sum_{i=1}^k S_iS_i^{*}=I$.

Fix $1 \leq j \leq k$. According to the assumption of the Cuntz-Krieger condition (\ref{def:CuntzKriegerCondition2}), there exists $J \subseteq \{1,2,\cdots,k\}$ such that $f(X_j)=\bigcup_{i \in J}X_i$.
We define $A(j,i)$ for $1 \leq i \leq k$ as follows:
\begin{equation*}
A(j,i)=
\begin{cases}
1, & i \in J,\\
0, & o.w.
\end{cases}
\end{equation*}
Consider the $k \times k$ matrix $A=(A(j,i))$.
For every $1 \leq i \leq k$, by the assumption of the Cuntz-Krieger condition (\ref{def:CuntzKriegerCondition1}), there exists $1 \leq j \leq k$ satisfying $X_i \subseteq f(X_j)$ and $i \in J$ which implies that $A(j,i)=1$.
It follows that no row of $A$ is zero.
By the definition of $A$, it holds that
\begin{align*}
S_j^{*}S_j=T_jT_j^{*}
=\sum_{i \in J}S_i S_i^{*}
=\sum_{i=1}^k A(j,i)S_i S_i^{*}
\end{align*}
for $1 \leq j \leq k$.
Therefore, $S_1,S_2,\cdots,S_k$ satisfy the relation of the Cuntz-Krieger algebra $\O_A$ and there exists a surjective $*$-homomorphism $\phi:\O_A \to C^{*}(T_1,T_2,\cdots,T_k)=C^{*}(S_1,S_2,\cdots,S_k)$.
\end{proof}

\subsection{Transformations of dynamical systems}

Finally, we discuss transformations of dynamical systems which is generalization of transformation of the Collatz conjecture in subsection \ref{subsection:FC}. Let $\Sigma (\neq \emptyset) \subseteq X$ and $P$ be the first return map for $f$ on $\Sigma$. Assume that the domain of $P$ coincides with $\Sigma$.

\begin{proposition}[Sufficient condition] \label{prop:transS}
Assume that $\orb(x;f) \cap \Sigma \neq \emptyset$ for every $x \in X$.
If $x \sim y$ by $P$ for every $x,y \in \Sigma$, then $x \sim y$ by $f$ for every $x,y \in X$.
\end{proposition}
\begin{proof}
Let $x,y \in X$. By the assumption, we can pick
\begin{align*}
z &\in \orb(x;f) \cap \Sigma \neq \emptyset,\\
w &\in \orb(y;f) \cap \Sigma \neq \emptyset,
\end{align*}
and then $x \sim z$, $y \sim w$ by $f$. Since $z,w \in \Sigma$, it holds that $z \sim w$ by $P$, that is, $\orb(z;P) \cap \orb(w;P) \neq \emptyset$. It follows that $\orb(z;f) \cap \orb(w;f) \neq \emptyset$. We obtain $z \sim w$ by $f$ and this implies $x \sim y$ by $f$.
\end{proof}

\begin{proposition}[Necessary condition] \label{prop:transN}
Assume that there exists a periodic point $x \in X$ for $f$ such that $x \in \Sigma$ and $\orb(x;P)=\orb(x;f) \cap \Sigma$. 
If $y \sim z$ by $f$ for every $y,z \in X$, then $y \sim z$ by $P$ for every $y,z \in \Sigma$.
\end{proposition}
\begin{proof}
Let $y \in \Sigma$. By the assumption, it holds that $y \sim x$ by $f$ and $f^n(y)=x$ for some $n \in \N$. Then, $P^n(y) \in \orb(x;f) \cap \Sigma=\orb(x;P)$. It follows that $x \sim y$ by $P$ for every $y \in \Sigma$ and so we obtain the conclusion.
\end{proof}

\begin{example}
\begin{enumerate}[(i)]
\item ($qx{+}1$-function for $q=5$, Mersenne numbers)
For $q=5$, let $N_1=\{n \in \N \mid n \equiv 1,3,7,9 \pmod{10}\}$ and $N_2=f_5(N_1)=\{n \in \N \mid n \equiv 6,16,36,46 \pmod{50}\}$. Consider the first-return map for $f_5$ on $N_1 \cup N_2$, denoted by $P$.
It follows that $P(N_1)=N_2$.
Since $\{n \pmod{50} \mid n \in \N, \gcd(n,10)=2\}=\{2^k \pmod{50} \mid 1 \le k \le 20\}$ is a cyclic group generated by $2$ under the operation of multiplication, for every $n \in N_1 \cup N_2$, there exists $k \in \N$ such that $2^k n \in N_2$.
Therefore, we can obtain that $P(N_2)=N_1 \cup N_2$, and then, $P$ satisfies the Cuntz-Krieger condition.

Let $\H$ be a Hilbert space with $\dim{\H}=\aleph_0$ and $\{e_n\}_{n \in N_1 \cup N_2}$ be a C.O.N.S. for $\H$. Define $T_1, T_2 \in \B(\H)$ by:
\begin{align*}
T_1e_n&=
\begin{cases}
e_{P(n)}, & n \in N_1,\\
0, & n \in N_2,
\end{cases}
\\
T_2e_n&=
\begin{cases}
0, & n \in N_1,\\
e_{P(n)}, & n \in N_2,
\end{cases}
\end{align*}
and $S_1,S_2 \in \B(\H)$ as $S_1=T_1^{*}T_2^{*}, S_2=T_2^{*}$.
It follows that, $C^{*}(T_1,T_2)=C^{*}(S_1,S_2)$ is isomorphic to the Cuntz algebra $\O_2$.

For every $n \in \N$, $\orb(n;f_5)$ includes an odd number which is relatively prime to $5$. Thus, $\orb(n;f_5) \cap \left(N_1 \cup N_2\right) \neq \emptyset$ for every $n \in \N$.
Furthermore, $1$ is a periodic point for $f_5$ such that $\orb(1;P)=\orb(1;f_5) \cap \left(N_1 \cup N_2\right)$.
According to proposition \ref{prop:transS}, \ref{prop:transN}, $n \sim m$ by $f_5$ for every $n,m \in \N$ if and only if $n \sim m$ by $P$ for every $n,m \in N_1 \cup N_2$.

For $k>2$ and the Mersenne number $q=2^k-1$, let $N_1=\{n \in \N \mid \text{$n$ is odd and } \exists l \in \N, 2n \equiv 2^l \pmod{2q^2}\}$, $N_2=f_q(N_1)$.
Consider the first-return map for $f_q$ on $N_1 \cup N_2$, denoted by $P$.
It follows that $P(N_1)=N_2$.
By definitions of $f_q$ and $N_1$,
\begin{align*}
&N_2=f_q(N_1) \subseteq\\
&\{n \in \N \mid n \equiv 1+q,1+3q,1+5q,\cdots,1+(2q-1)q \pmod{2q^2}\}.
\end{align*} 
Since $1+q=2^k$ where $k>2$, $\binom{1+q}{l}$ is even number for $1<l<q$. Therefore,
\begin{align*}
(1+q)^{1+q}&=\sum_{i=0}^{1+q}\binom{1+q}{i}q^i\\
&\equiv 1+(1+q)q+(1+q)q^q+q^{1+q} \pmod{2q^2}\\
&\equiv 1+q+q^2+q^{1+q} \pmod{2q^2}\\
&\equiv 1+q+(1+q^{q-1})q^2 \pmod{2q^2}\\
&\equiv 1+q \pmod{2q^2}.
\end{align*}
Conversely, if $(1+q)^l \equiv (1+q) \pmod{2q^2}$ where $l \in \N\backslash\{1\}$, then
\begin{align*}
0 &\equiv (1+q)^l-(1+q) \pmod{2q^2}\\
&\equiv (1+q)((1+q)^{l-1}-1) \pmod{2q^2}\\
&\equiv 2^k\left( \sum_{i=1}^{l-1} \binom{l-1}{i}q^i \right) \pmod{2q^2}\\
&\equiv 2^k(l-1)q \pmod{2q^2} .
\end{align*}
It follows that $l-1$ is a multiple of $q$. Thus, $1+q, (1+q)^2, \cdots, (1+q)^q \pmod{2q^2}$ are different from each other.
Since $\{n \in \N \mid n \equiv 1+q,1+3q,1+5q,\cdots,1+(2q-1)q \pmod{2q^2}\}$ is multiplicatively closed, it holds that $\{(1+q)^l \pmod{2q^2} \mid 1 \le l \le q\}=\{1+q,1+3q,1+5q,\cdots,1+(2q-1)q \pmod{2q^2}\}$.
For every $m \in \N$,
\begin{align*}
2(1+q)^{2m} \equiv 2\sum_{i=0}^{2m}\binom{2m}{i}q^i
\equiv \sum_{i=0}^{2m}\binom{2m}{i}2q^i
\equiv 2+4mq \pmod{2q^2}.
\end{align*}
By $1+q=2^k$ and the former calculation, $1+2mq \in N_1$. Thus,
$\{n \in \N \mid n \equiv 1 \pmod{2q}\} \subseteq N_1$
and
$\{n \in \N \mid n \equiv 1+q \pmod{2q^2}\} \subseteq N_2$.
Hence, for every $n \in N_1 \cup N_2$, there exists $m \in \N$ such that $2^m n \in N_2$.
Therefore, we can obtain that $P(N_2)=N_1 \cup N_2$, and then, $P$ satisfies the Cuntz-Krieger condition.

Let $\H$ be a Hilbert space with $\dim{\H}=\aleph_0$ and $\{e_n\}_{n \in N_1 \cup N_2}$ be a C.O.N.S. for $\H$. Define $T_1, T_2 \in \B(\H)$ by:
\begin{align*}
T_1e_n&=
\begin{cases}
e_{P(n)}, & n \in N_1,\\
0, & n \in N_2,
\end{cases}
\\
T_2e_n&=
\begin{cases}
0, & n \in N_1,\\
e_{P(n)}, & n \in N_2,
\end{cases}
\end{align*}
and $S_1,S_2 \in \B(\H)$ as $S_1=T_1^{*}T_2^{*}, S_2=T_2^{*}$.
It follows that, $C^{*}(T_1,T_2)=C^{*}(S_1,S_2)$ is isomorphic to the Cuntz algebra $\O_2$.

For every $n \in \N$, $\orb(n;f_q)$ includes an even number $qm+1$ for some odd number $m \in \N$.
Let $x \in \N$ be the positive number satisfying that $2^{-x}(qm+1)$ is an odd number.
Since $1+q^2 \equiv (1+q)^l \pmod{2q^2}$ for some $l \in \N$ and $1+q=2^k$, there exists $l' \in \N$ such that $l'>x$ and $2^{l'} \equiv 1+q^2 \pmod{2q^2}$.
It follows that $2^{l'+1} \equiv 2+2q^2 \equiv 2 \pmod{2q^2}$.
For $qm+1$, there exists $l'' \in \N$ such that $qm+1 \equiv (1+q)^{l''} \pmod{2q^2}$. Then,
\begin{align*}
2 \cdot 2^{-x}(qm+1) &\equiv 2^{l'+1-x}(qm+1) \pmod{2q^2}\\
&\equiv 2^{l'+1-x}(1+q)^{l''} \pmod{2q^2}\\
&\equiv 2^{l'+1-x}2^{kl''} \equiv 2^{l'+1-x+kl''} \pmod{2q^2}.
\end{align*}
Hence, $2^{-x}(qm+1) \in \orb(n;f_q) \cap N_1$, and so $\orb(n;f_q) \cap \left(N_1 \cup N_2\right) \neq \emptyset$.
Furthermore, $1$ is a periodic point for $f_q$ satisfying $\orb(1;P)=\orb(1;f_q) \cap \left(N_1 \cup N_2\right)$.
According to proposition \ref{prop:transS}, \ref{prop:transN}, $n \sim m$ by $f_q$ for every $n,m \in \N$ if and only if $n \sim m$ by $P$ for every $n,m \in N_1 \cup N_2$.

\item ($3x{+}d$-function)
Let $d>0$ be an odd integer and $k\ge0$ be the integer satisfying that $3^k \mid d$ and $3^{k+1} \nmid d$.
Define $N_1=\{n \in \N \mid n \equiv 3^k, 5\cdot3^k \pmod{6\cdot3^k}\}$ and $N_2=f_d(N_1)=\{n+d \mid n \in \N, n \equiv 3\cdot3^k, 15\cdot3^k \pmod{18\cdot3^k}\}$.
Consider the first-return map for $f_d$ on $N_1 \cup N_2$, denoted by $P$.
It follows that $P(N_1)=N_2$.
Since $\{n \pmod{18\cdot3^k} \mid n \in \N, \gcd(n,18\cdot3^k)=2\cdot3^k\}=\{2^l\cdot3^k \pmod{18\cdot3^k} \mid 1 \le l \le 6\}$, for every $n \in N_1 \cup N_2$, there exists sufficiently large $m \in \N$ such that $2^m n \in N_2$.
Therefore, we can obtain $P(N_2)=N_1 \cup N_2$, and then, $P$ satisfies the Cuntz-Krieger condition.

Let $\H$ be a Hilbert space with $\dim{\H}=\aleph_0$ and $\{e_n\}_{n \in N_1 \cup N_2}$ be a C.O.N.S. for $\H$. Define $T_1, T_2 \in \B(\H)$ by:
\begin{align*}
T_1e_n&=
\begin{cases}
e_{P(n)}, & n \in N_1,\\
0, & n \in N_2,
\end{cases}
\\
T_2e_n&=
\begin{cases}
0, & n \in N_1,\\
e_{P(n)}, & n \in N_2,
\end{cases}
\end{align*}
and $S_1,S_2 \in \B(\H)$ as $S_1=T_1^{*}T_2^{*}, S_2=T_2^{*}$.
It follows that, $C^{*}(T_1,T_2)=C^{*}(S_1,S_2)$ is isomorphic to the Cuntz algebra $\O_2$.

For every $n \in \N$, $\orb(n;f_d)$ includes an odd number which is in $N_1$. Thus, $\orb(n;f_d) \cap \left(N_1 \cup N_2\right) \neq \emptyset$ for every $n \in \N$.
Furthermore, $d$ is a periodic point for $f_d$ such that $\orb(d;P)=\orb(d;f_d) \cap \left(N_1 \cup N_2\right)$.
According to proposition \ref{prop:transS}, \ref{prop:transN}, $n \sim m$ by $f_d$ for every $n,m \in \N$ if and only if $n \sim m$ by $P$ for every $n,m \in N_1 \cup N_2$.
\end{enumerate}
\end{example}

\section*{Acknowledgements}
I would like to thank my thesis supervisor, Professor Hiroki Matui (Chiba University) for helpful discussions in the seminar of the study, and appropriate advices from the first draft of the paper.


\bibliographystyle{plain}

\bibliography{bib}

\end{document}